\documentclass[english]{amsart}

\usepackage{amsmath}
\usepackage{amssymb}
\usepackage{amsthm}
\usepackage{amscd}
\usepackage{babel}
\usepackage[latin1]{inputenc}
\usepackage{graphicx}

\newtheorem{teor}{Theorem}

\newtheorem{prop}{Proposition}

\newtheorem{lem}{Lemma}

\title{About the congruence $\sum_{k=1}^n k^{f(n)} \equiv 0 \textrm{ (mod n)} $}

\author{Jos\'{e} Mar\'{i}a Grau}
\address{Departamento de Matemáticas, Universidad de Oviedo\\ Avda. Calvo Sotelo, s/n, 33007 Oviedo, Spain}
\email{grau@uniovi.es}

\author{Antonio M. Oller-Marc\'{e}n}
\address{Centro Universitario de la Defensa\\ Ctra. Huesca s/n, 50090 Zaragoza, Spain} \email{oller@unizar.es}
\begin{document}
\maketitle

\begin{abstract}
In this paper we characterize, in terms of the prime divisors of $n$, the pairs $(k,n)$ for which $n$ divides $\sum_{j=1}^n j^{k}$. As an application, we study the sets $\mathcal{M}_f :=\{n: n \textrm{ divides } \sum_{j=1}^n j^{f(n)} \}$ for some choices of $f$.
\end{abstract}

\section{Introduction}

In the literature on power sums
$$S_k(n):=\displaystyle{\sum_{j=1}^{n} j^k}$$
the following congruence is well known

\begin{prop} Let $p$ be a prime and let $k>0$ be an integer. Then, we have:
$$
S_k(p)\equiv\begin{cases}
-1 & {\rm{if}}\ p-1\mid k\\
0 & {\rm{if}}\ p-1\nmid k\\
\end{cases}\ (mod\ p).
$$
\end{prop}
\begin{proof} See (\cite{HARDY}) for the standard proof using primitive
roots, or (\cite{MAC}) for a recent elementary proof.
\end{proof}

The following proposition gives a more general result for $S_{k}(n)$.

\begin{prop} (Carlitz-von Staudt, 1961, \cite{CARLITZ}). Let $k>1$ and $n$ be positive integers with $n$ even, then
$$
S_k(n)\equiv - \sum_{\substack{p \mid n\\ p-1 \mid k}}\frac{n}{p}\ (mod\ n).
$$
\end{prop}

These results motivate an interest in studying $S_{k}(n)$ (mod $n$) and, more generally, in studying $S_{f(n)}(n)$ (mod $n$) for different arithmetic functions $f$. Thus, if $p-1 \mid f(p)$, for every prime $p$, we have that the congruence $S_{f(n)}(n) \equiv -1$ (mod $n$) holds for every $n=p$ prime and it is interesting to find the composite numbers which also satisfy it. In this direction we have the \emph{Giuga numbers} (see \cite{BOR}), which are composite numbers such that $S_{\phi(n)}(n) \equiv -1$ (mod $n$), the \emph{strong Giuga numbers}, which are composite numbers such that $S_{n-1}(n) \equiv -1$ (mod $n$) (Giuga's conjecture \cite{GIUGA} states that there are no strong Giuga numbers), or the \emph{$K$-strong Giuga numbers}, which are composite numbers such that $S_{K(n-1)}(n) \equiv -1$ (mod $n$) (see \cite{nos1}). 

In this paper we characterize, in terms of the prime divisors of $n$, the pairs $(k,n)$ for which $n$ divides $S_{k}(n)$. This characterization is given in the following theorem.
\begin{teor}
Let $n$ and $k$ be any integers. Then, $S_k(n)\equiv 0$ (mod $n$) if and only if one of the following holds:
\begin{itemize}
\item[i)]
$n$ is odd and $p-1 \nmid  k $ for every $p$ prime divisor of $n$.
\item[ii)]
$n$ is a multiple of 4 and $k>1$ is odd.
\end{itemize}
\end{teor}

Moreover, inspired in Giuga's ideas we investigate the congruence $S_{f(n)}(n) \equiv 0$ (mod $n$) for some functions $f$. This work started in \cite{nos2}) with the case $f(n)=\frac{(n-1)}{2}$. It will be of special interest the case of arithmetic functions $f$ such that $p-1 \nmid f(p)$ for any prime $p$. %como es el caso de Euler Totien function and the lambda Carmichael function.
%En el caso de la congruencia  $S_{f(n)}(n) \equiv 1 \textrm{ (mod n)}$ son muchas las funciones para las cuales se verifica para los mismos valores que para $f(n)=n$.

Let $f:\mathbb{N}\longrightarrow\mathbb{N}$ be a function. In what follows we will consider the following subset of $\mathbb{N}$ associated to $f$:
$$\mathcal{M}_f :=\{n: n\ \textrm{divides}\ S_{f(n)}(n) \}$$

We have studied the sets $\mathcal{M}_f$ in the affine case ($f(n)=a n +b $) and in some cases such that $\mathcal{M}_f$ contains the set of prime numbers. We have characterized the elements of these sets and, in some cases, we have computed their asymptotic density.

\section{Sum of the $n$ initial $k$-th powers modulo $n$}

Recall that, given two integers $n$ and $k$, we define $S_k(n):=\displaystyle{\sum_{j=1}^{n} j^k}$. In this section we present the main results of the paper. In particular we will characterize the pairs $(n,k)$ such that $n$ divides $S_k(n)$. If $k=0$, clearly $S_k(n)=n$ and there is no problem to study. Thus, in what follows we will assume $k\neq 0$.

We will start this section with a technical lemma.

\begin{lem}
Let $k>0$ be any integer and let $p$ be a prime such that $p-1 \mid k$. Then, for every $m>0$:
$$\sum_{j=1}^{p^m}j^k\equiv -p^{m-1}\ (mod\ p^m).$$
\end{lem}
\begin{proof}
Clearly $\displaystyle{\sum_{j=1}^{p^m}j^k\equiv \sum_{\substack{1\leq j\leq p^m\\ (j,p^m)=1}} j^k}$ (mod $p^m$), with this latter sum consisting of $p^{m-1}(p-1)$ summands. Since $p-1 \mid k$, if $a$ and $b$ are such that $p \nmid a,b$ then $a^k-b^k\equiv 0$ (mod $p$). This implies that these summands are the elements of the arithmetic sequence $\{1,p+1,\dots,p(p^{m-1}-1)+1\}$, where every element appears exactly $p-1$ times. Consequently:
\begin{align*}
\sum_{j=1}^{p^m}j^k&\equiv \sum_{\substack{1\leq j\leq p^m\\ (j,p^m)=1}} j^k=(p-1)\sum_{i=1}^{p^{m-1}-1} 1+ip=(p-1)\frac{p^{m-1}[1+(1+p^m-p)]}{2}\\&\equiv (p-1)p^{m-1}\equiv -p^{m-1}\ \textrm{(mod $p^m$)}.
\end{align*}
\end{proof}

With the help of this lemma we can prove the following proposition
\begin{prop}
Let $n$ be an odd integer and $k$ be any integer. Then $n$ divides $S_k(n)$ if and only if $\gcd(k,p-1)<p-1$ for every $p$, prime divisor of $n$.
\end{prop}
\begin{proof}
Put $n=p_1^{r_1}\cdots p_s^{r_s}$ the prime decomposition of $n$.

Assume that there exists $i\in\{1,\dots,s\}$ such that $p_i-1=\gcd(k,p_i-1)$; i.e., such that $p_i-1|k$. Then, since $S_k(n)=\displaystyle{\sum_{j=1}^nj^k\equiv \frac{n}{p_i^{r_i}}\sum_{j=1}^{p_i^{r_i}}j^k}$ (mod $p_i^{r_i}$), it follows by the previous lemma that $S_k(n)\equiv\frac{n}{p_i^{r_i}}(p_i-1)p_i^{r_i-1}\not\equiv 0$ (mod $p_i^{r_i}$) and $n$ does not divide $S_k(n)$.

Conversely, let us assume that $\gcd(k,p_i-1)<p_i-1$ for every $i$. We claim that $\displaystyle{\sum_{\substack{(j,n)=1\\1\leq j\leq n}} j^k\equiv 0}$ (mod $n$).
Let $i\in\{1,\dots,s\}$ and let $\alpha$ be a generator of $\mathcal{U}\left(\mathbb{Z}/p_i^{r_i}\mathbb{Z}\right)$, we put $x=\alpha^k$. If $x-1$ was not a unit, then $p_i\mid x-1=\alpha^k-1$. Since $p_i\mid \alpha^{p_i-1}-1$ it follows that $\alpha^d\equiv 1$ (mod $p_i$) with $d=\gcd(k,p_i-1)$. Consequently $\alpha^{dp_i^{r_i-1}}\equiv 1$ (mod $p_i^{r_1}$) which is impossible since $dp_i^{r_i-1}<\varphi(p_i^{r_1})$. Thus:
$$\sum_{\substack{(j,p_i^{r_1})=1\\1\leq j\leq p_i^{r_i}}}j^k\equiv\frac{x^{\varphi(p_i^{r_i})+1}-x}{x-1}\equiv 0\ \textrm{(mod $p_i^{r_i}$)}.$$
Moreover, since $\displaystyle{\sum_{\substack{(j,n)=1\\1\leq j\leq n}}j^k\equiv \varphi\left(\frac{n}{p_i^{r_i}}\right)\sum_{\substack{(j,p_i^{r_i})=1\\1\leq j\leq p_i^{r_i}}}j^k\equiv 0}$ (mod $p_i^{r_i}$) and this holds for every $i$, the claim follows.

Now, let $d$ be any divisor of $n$. We have that $\displaystyle{\sum_{\substack{(j,n)=d\\1\leq j\leq n}}j^k=d^k\sum_{\substack{(j,n/d)=1\\1\leq j\leq n/d}}j^k}$ so it is enough to apply the previous considerations and to sum over $d$ in order to complete the proof.
\end{proof}

Now we will turn to the even case. This is done in the following propositions.

\begin{prop}
Let $n$ be an integer with $n\equiv 2$ (mod 4) and let $k$ be any integer. Then:
\begin{itemize}
\item[i)]
$S_k(n)\not\equiv 0$ (mod $n$).
\item[ii)]
$S_k(n)\equiv 0$ (mod $\frac{n}{2}$) if and only if $\gcd(k,p-1)<p-1$ for every $p$, odd prime divisor of $n$.
\end{itemize}
\end{prop}
\begin{proof}
\begin{itemize}
\item[i)]
We have that $j^k\equiv 0,1$ (mod 2) if $j$ even or odd respectively. This implies that $S_k(n)\equiv\frac{n}{2}\not\equiv 0$ (mod 2).
\item[ii)]
Proposition 3 implies that $S_{k}\left(\frac{n}{2}\right)\equiv 0$ (mod $\frac{n}{2}$) if and only if $\gcd(k,p-1)<p-1$ for every $p$ prime divisor of $\frac{n}{2}$. Since $S_k(n)\equiv 2S_{k}\left(\frac{n}{2}\right)$ (mod $\frac{n}{2}$), the result follows.
\end{itemize}
\end{proof}

\begin{prop}
Let $n$ be a multiple of 4 and let $k$ be an odd integer. Then:
\begin{itemize}
\item[i)]
If $k=1$, $S_k(n)\not\equiv 0$ (mod $n$).
\item[ii)]
If $k>1$, $S_k(n)\equiv 0$ (mod $n$).
\end{itemize}
\end{prop}
\begin{proof}
The first part is obvious. For the second part, put $n=2^mn'$ with $m>1$ and $n'$ being odd.

Since $k$ is odd it follows that $\gcd(k,p-1)<p-1)$ for every $p$ prime divisor of $n'$ and Proposition 1 implies that $S_k(n)\equiv 0$ (mod $n'$). On the other hand, we have that $S_k(n)\equiv 2^mS_k(n')\equiv 0$ (mod $n'$).

Moreover, $S_k(n)\equiv n'S_{k}(2^m)$ (mod $2^m$). Now, $k>1$ being odd, if $j\in\{1,\dots,2^m-1\}$ we have that $j^k\equiv -(2^m-j)^k$ (mod $2^m$) so $S_k(2^m)\equiv (2^{m-1})^k\equiv 0$ (mod $2^m$) and the result follows.
\end{proof}

\begin{prop}
Let $n$ be a multiple of 4 and let $k$ be an even integer. Then:
\begin{itemize}
\item[i)]
$S_k(n)\not\equiv 0$ (mod $n$).
\item[ii)]
$S_k(n)\equiv 0$ $\left(\ \textrm{mod}\ \frac{n}{2}\right)$ if and only if $\gcd(k,p-1)<p-1$ for every $p$, odd prime divisor of $n$.
\end{itemize}
\end{prop}
\begin{proof}
\begin{itemize}
\item[i)]
Put $n=2^mn'$ with $m>1$ and $n'$ being odd.

Again $S_k(n)\equiv n'S_k(2^m)$ (mod $2^m$) and, since $k$ is even, we have that $j^k\equiv (2^m-j)^k$ (mod $2^m$). This implies that $S_k(2^m)\equiv 2S_k(2^{m-1})$ (mod $2^m$). This allows us to reason inductively and conclude that:
$$S_k(2^m)\equiv 0 \Leftrightarrow S_k(1)\equiv 0\ \textrm{(mod 2)}.$$
Since the latter is false the result follows.
\item[ii)]
By Proposition 1 we have that $S_k(n)\equiv 2^mS_k(n')\equiv 0$ (mod $n'$) if and only if $\gcd(k,p-1)<p-1$ for every $p$ prime divisor of $n'$. Now $S_k(n)\equiv n'S_{k}(2^m)$ (mod $2^m$). Clearly $S_k(1)=1$ and if $a>1$ we have that $S_k(2^a)\equiv 2S_k(2^{a-1})$ (mod $2^a$). This implies that $S_k(2^m)\equiv 2^{m-1}$ (mod $2^m$) and, consequently, $S_k(n)\equiv n'S_k(2^m)\equiv n'2^{m-1}=\frac{n}{2}$ (mod $2^m$) and the proof is complete.
\end{itemize}
\end{proof}

All the previous work can be summarized in the following theorem.

\begin{teor}\label{TI}
Let $n$ and $k$ be any integers. Then, $S_k(n)\equiv 0$ (mod $n$) if and only if one of the following holds:
\begin{itemize}
\item[i)]
$n$ is odd and $ p-1 \nmid k$ for every $p$ prime divisor of $n$.
\item[ii)]
$n$ is a multiple of 4 and $k>1$ is odd.
\end{itemize}
\end{teor}

We have just characterized the pairs $(n,k)$ such that $n$ divides $S_k(n)$. It follows immediately from this characterization that, given $n \in \mathbb{N}$ the complement of the set $\mathcal{W}_n:=\{k: S_{k}(n) \equiv 0\textrm{ (mod }n) \}$ is:
$$
\mathbb{N} \setminus \mathcal{W}_{n}=\begin{cases}
2\mathbb{N}\cup\{1\} & \textrm{if}\ n\equiv 0\ \textrm{(mod 4)},\\
\mathbb{N} & \textrm{if}\ n\equiv 2\ \textrm{(mod 4)},\\
\bigcup_{p \mid n} (p-1)\mathbb{N} & \textrm{if}\ n\equiv 1,3\ \textrm{(mod 4)}.
\end{cases}
$$
In the same way, for every $k \in \mathbb{N}$ the complement of $\mathcal{H}_k:=\{n : S_{k}(n) \equiv 0\textrm{ (mod }n) \}$ consists of a finite union of arithmetic sequences. Namely, if we denote $\mathcal{P}_k:=\{p\ \textrm{odd prime}: p-1 \mid k\}$, we have that:

$$
\mathbb{N} \setminus\mathcal{H}_{k}=\begin{cases}
2\mathbb{N} & \textrm{if}\ k=1,\\
\bigcup_{p\in\mathcal{P}_k}2p(2\mathbb{N}+1) & \textrm{if}\ k>1\ \textrm{is odd},\\
\bigcup_{p\in\mathcal{P}_k}p\mathbb{N} & \textrm{if}\ k\ \textrm{is even}.
\end{cases}
$$

Now, we could consider cases when $k=f(n)$ depends on $n$ and then we will be interested in characterizing the values of $n$ such that $S_{f(n)}(n)\equiv 0$ (mod $n$). This will be done in the following sections for various choices of the function $f$.

\section{The affine case}
% 451924450280

%317620
In this section we will focus in the case when $f$ is a affine function; i.e., $f(n)=an+b$. In what follows we will denote
by $f_{a,b}(n):= an + b$. Recall that we defined $\mathcal{M}_f :=\{n:n\ \textrm{divides}\ S_{f(n)}(n) \}$. In what follows it will be easier to characterize the complement $\mathbb{N}\setminus\mathcal{M}_f$ instead of $\mathcal{M}_f$ itself.

Let us introduce some notation. Given $(a,b) \in \mathbb{N}\times \mathbb{Z}$, we will consider the set:
$$\mathcal{P}_{a,b}:=\{p\ \textrm{odd prime}: b\equiv 0\ \textrm{(mod $\gcd(ap,p-1)$)}\}.$$
and if $(a,b,p) \in \mathbb{N}\times \mathbb{Z} \times \mathcal{P}_{a,b} $ we define
$$\Xi(a,b,p):=a^{-1}\min \{  x\in \mathbb{N} : x \equiv 0\textrm{ (mod }ap), x \equiv -b\textrm{ (mod }(p-1))\}.$$
%Nótese que $\Xi(a,b,p)$ está bien definido siempre que $p \in \mathcal{P}_{a,b}$
With this notation in mind we can prove the following result  \cite{NIVEN}.

\begin{teor}
Let $(a,b) \in \mathbb{N}\times \mathbb{Z}$. Then:
\begin{itemize}
\item[i)] If $a$ and $b$ are even,
$$\mathbb{N} \setminus \mathcal{M}_{f_{a,b}}=2 \mathbb{N}\ \cup\bigcup_{p \in \mathcal{P}_{a,b}}\{ \Xi(a,b,p)+\frac{s}{a} {\rm{lcm}}(ap,p-1):s \in \mathbb{N} \}.$$
%with $\Xi(a,b,p)$ a positive multiple of $p$.
\item[ii)]  If $a$ and $b$ are odd,
$$\mathbb{N} \setminus \mathcal{M}_{f_{a,b}}=\{n: n\equiv 2\ {\rm{(mod\ 4)}}\}\ \cup\bigcup_{p \in \mathcal{P}_{a,b}}\{ \Xi(a,b,p)+\frac{s}{a} {\rm{lcm}}(ap,p-1):s \in \mathbb{N} \}.$$
%with $\Xi(a,b,p)$ a positive multiple of $p$.
\item[iii)] If $a$ is even and $b$ is odd, then
$$\mathbb{N} \setminus \mathcal{M}_{f_{a,b}}=\{n:n\equiv 2\ \rm{(mod\ 4)}\}.$$
\item[iv)] If $a$ is odd and $b$ is even, then
$$\mathbb{N} \setminus \mathcal{M}_{f_{a,b}}=2\mathbb{N}.$$
\end{itemize}
\end{teor}
\begin{proof}
We will give a complete proof of i), the other cases being analogous.

First observe that $a$ and $b$ being even, then $f_{a,b}(n)$ is even. Consequently, by Theorem \ref{TI}, we have that $2\mathbb{N}\subseteq\mathbb{N}\setminus\mathcal{M}_{f_{a,b}}$.

Now, assume that $n\not\in\mathcal{M}_{f_{a,b}}$ is odd. Then, by Theorem \ref{TI} again, there must exist an odd prime $p \mid n$ such that $p-1\mid an+b$. Since $an\equiv 0$ (mod $ap$) and $an\equiv -b$ (mod $p-1$) it readily follows that $an\in\{A+(s)\textrm{lcm}(ap,p-1):s\in\mathbb{N}\}$ with $A=\min \{  x\in \mathbb{N} : x \equiv 0\textrm{ (mod }ap), x \equiv -b\textrm{ (mod }(p-1))\}$. Since it is obvious that $A$ is a multiple of $ap$ we have that $n\in\{\frac{A}{a}+\frac{(s)}{a}\textrm{lcm}(ap,p-1):s\in\mathbb{N}\}$ with $\frac{A}{a}=\Xi(a,b,p)$ by definition. To finish the proof it is enough to observe that if $ap\mid an$ and $p-1\mid an+b$, then $p\in\mathcal{P}_{a,b}$ as claimed.
\end{proof}

Here and throughout, we denote by $\delta(A)$ (resp. $\underline{\delta}(A)$, $ \overline{\delta}(A)$) the asymptotic (resp. upper, lower asymptotic) density of an
integer sequence $A$. We will be interested in computing the asymptotic density of the sets $\mathcal{M}_{f_{a,b}}$, at least for some particular values of $a$ and $b$. To do so we must first show that this density exists and the following lemma will be our main tool.

\begin{lem}
Let $\mathcal{A}:=\{a_k\}_{k\in\mathbb{N}}$ and $\{c_k\}_{k\in\mathbb{N}}$ be two sequences of positive integers and $\mathcal{B}_{k}:=\{a_k+ (s-1)c_k: s \in \mathbb{N} \}$. If $\displaystyle{\sum_{k=1}^{\infty}\frac{1}{c_{k}}}$ is convergent and $\mathcal{A}$ has zero asymptotic density, then $\displaystyle{\bigcup_{k=1}^{\infty}\mathcal{B}_{k}}$ has an asymptotic density with:
$$\delta(\bigcup_{k=1}^{\infty}\mathcal{B}_{k})=\lim_{n\rightarrow\infty}\delta(\bigcup_{k=1}^{n}\mathcal{B}_{k})$$
and
$$ \delta(\bigcup_{k=1}^{\infty}\mathcal{B}_{k})-\delta(\bigcup_{k=1}^{n}\mathcal{B}_{k})\leq \sum_{i=n+1}^\infty \frac{1}{c_i}.$$
\end{lem}
\begin{proof}
Let us denote $B_{n}:=\bigcup_{k=n+1}^{\infty}\mathcal{B}_{k}$ and $\vartheta(n,N):=\textrm{card}([0,N]\cap B_n)$. Then:
$$\vartheta(n,N) \leq \textrm{card}([0,N]\cap \mathcal{A})+ N \sum_{k=n+1}^{\infty}\frac{1}{c_{k}}.$$
From this, we get:
$$\bar{\delta}(B_{n})=\lim\sup\frac{\vartheta(n,N)}{N}\leq \lim \sup\frac{ \textrm{card}([0,N]\cap \mathcal{A})}{N} + \sum_{k=n+1}^{\infty}\frac{1}{c_{k}}= \sum_{k=n+1}^{\infty}\frac{1}{c_{k}}.$$

Now, for every $n$, $\displaystyle{\bigcup_{k=1}^{n}\mathcal{B}_{k}}$
has an asymptotic density and the sequence $\delta_{n}:=\delta\left(\displaystyle{\bigcup
_{k=1}^{n}\mathcal{B}_{k}}\right)$ is non-decreasing and bounded (by 1), thus convergent. Consequently:

\begin{align*}
\delta\left(\bigcup_{k=1}^{n}\mathcal{B}_{k}\right) & \leq \underline{\delta} \left(\bigcup
_{k=1}^{\infty}\mathcal{B}_{k}\right)\leq \overline{\delta} \left(\bigcup_{k=1}^{\infty}
\mathcal{B}_{k}\right)=\overline{\delta}\left(\bigcup_{k=1}^{n}\mathcal{B}_{k}\cup B_{n}\right)\\
& \leq \delta \left(\bigcup_{k=1}^{n}\mathcal{B}_{k}\right)+\bar{\delta}(B_{n})\leq
\delta\left(\bigcup_{k=1}^{n}\mathcal{B}_{k}\right)+\sum_{k=n+1}^{\infty}\frac{1}{c_{k}},
\end{align*}

and considering that $\displaystyle{\sum_{k=n+1}^{\infty}\frac{1}{c_{j}}}$ converges to zero
it is enough to take limits in order to finish the proof.
\end{proof}

With the help of this lemma the following proposition is easy to prove.
\begin{prop}
For every $a,b \in \mathbb{N}$, $\mathcal{M}_{f_{a,b}}$ has an asymptotic density and it is a computable number.
\end{prop}
\begin{proof}
It is enough to see that $\mathbb{N} \setminus \mathcal{M}_{f_{a,b}}$ has an asymptotic density and that it is a computable number.

Cases ii) y iii) above are obvious. In cases i) and iv) it is enough to apply the previous lemma since $\mathbb{N} \setminus \mathcal{M}_{f_{a,b}}$ is a countable union of arithmetic sequences whose initial terms ($\Xi(a,b,p)$) form a set of zero asymptotic density, and the series $\displaystyle{\sum_{p\textrm{ prime}}\frac{a}{\textrm{lcm}(ap,p-1)}}$ is convergent.
\end{proof}

The rest of this section will be devoted to study $\delta(\mathcal{M}_{f_{_{a,b}}})$ in some particular cases. Namely, the cases $(a,b)=(1,b)$. When $b$ is even, $\mathcal{M}_{f_{_{1,b}}}$ is exactly the set of odd integers and its asymptotic density is $\frac{1}{2}$. The case when $b$ is odd is much more interesting. In particular we will see that, in this case, the asymptotic density of $\mathcal{M}_{f_{1,b}}$  is slightly  greater than $\frac{1}{2}$.

In the following lemma we give a more explicit description of the elements of $\mathbb{N}\setminus\mathcal{M}_{f_{1,b}}$ with odd b. This description will be useful to compute $\delta(\mathcal{M}_{f_{{1,b}}})$.

\begin{lem}
Let $n$ be an integer and let $b\in\mathbb{Z}$ be odd. Then $n\in \mathbb{N} \setminus \mathcal{M}_{f_{1,b}}$ if and only if $n\equiv 2$ (mod 4) or $n$ is odd and it is of the form $kp^2-kp-bp$ for some $p$ odd prime and $\frac{b}{p-1}<k\in\mathbb{Z}$. In other words, if $\mathcal{G}_p^b:=\mathbb{N} \cap \{-bp\ \textrm{(mod $p(p-1)$)}\}$, we have that:
$$\mathbb{N}\setminus \mathcal{M}_{f_{_{1,b}}}=\bigcup_{p\geq3} \mathcal{G}_p^b \cup \{(2\ \textrm{(mod 4)}\}.$$
\end{lem}
\begin{proof}
Let $n\in \mathbb{N} \setminus \mathcal{M}_{f_{1,b}}$. Then, by Theorem \ref{TI}, $n\equiv 2$ (mod 4) or it is odd and there exists $p$ prime divisor of $n$ such that $p-1$ divides $n+b$. Put $n=pm$, then $n+b=pm+b=(p-1)m+m-b$ so $p-1$ must divide $m-b$ and $m=k(p-1)+b$ for some $k$ and $n=pm=kp^2-kp-bp$ as claimed.

The converse is obvious due to Theorem \ref{TI} again.
\end{proof}

In order to compute $\delta(\mathcal{M}_{f_{_{1,b}}})=1-\delta(\mathbb{N} \setminus \mathcal{M}_{f_{1,b}})$  with the help of lemmata 2 and 3 and of the Principle of Inclusion and Exclusion it will be necessary to have a good criterion to determine when the intersection of $\mathcal{ G}_p^b$ for various odd primes $p$ is empty. Put
$$
\mathcal{R}_b:=\{m>2:\gcd (m,\phi(m))\textrm{ divides }b\}.
$$

Thus, we have the following result.

\begin{prop}
Let ${\mathcal P}$ be a finite set of primes and put $m:=\prod_{p\in {\mathcal P}} p$. Then $\bigcap_{p\in {\mathcal P}} {\mathcal G}_p^b$ is nonempty if and only if
$m \in \mathcal{R}_b$, where this set is defined above. If this is the case, then the set $\bigcap_{p\in {\mathcal P}} {\mathcal G}_p^b$ is an arithmetic progression of difference $\textrm{lcm}(m,\lambda(m))$.
\end{prop}
\begin{proof}
It is clear that $\bigcap_{p\in\mathcal{P}}\mathcal{G}_b^b$ is nonempty if and only if there exists $n$ such that $n\equiv -b$ (mod $p-1$) and $n\equiv 0$ (mod $p$) for every $p\in\mathcal{P}$. This happens if and only if there exists $n$ such that $n\equiv -b$ (mod $\lambda(m)$) and $n\equiv 0$ (mod $m$) and this set of congruences have a solution if and only if $\gcd(m,\lambda(m))$ divides $b$. To finish the proof it is enough to observe that, $m$ being square-free, $\gcd(m,\lambda(m))=\gcd(m,\phi(m))$ and apply the Chinese Remainder Lemma.
\end{proof}

To compute the density of the set $\mathbb{N} \setminus \mathcal{M}_{f_{1,b}}$ we consider  $3=p_1<p_2<\cdots$  the increasing sequence of all the odd primes and $k:=k(\varepsilon)$  minimal such that
$$
\sum_{j\ge k} \frac{1}{p_j(p_j-1)}<\varepsilon.
$$
Thus, with an error of at most $\varepsilon$, the density of the set $\mathbb{N} \setminus \mathcal{M}_{f_{1,b}}$ is the same as the density of
$\bigcup_{j<k} {\mathcal G}_{p_j}^b: $
$$\delta  \left( \bigcup_{j<k} {\mathcal G}_{p_j}^b \right)< \delta(\mathbb{N} \setminus \mathcal{M}_{f_{1,b}}) <\delta  \left( \bigcup_{j<k} {\mathcal G}_{p_j}^b \right)+\varepsilon$$

and, by the Principle of Inclusion and Exclusion,
$$
%\label{eq:s}
\delta  \left( \bigcup_{j<k} {\mathcal G}_{p_j}^b \right)=\sum_{s\ge 1} \sum_{1\le i_1<i_2<\cdots<i_s\le k-1}\frac{\varepsilon_{i_1,i_2,\ldots,i_s}}{\textrm{lcm}[p_{i_1}(p_{i_1}-1),\ldots, p_{i_s}(p_{i_s}-1)]},
$$
with  the coefficient $\varepsilon_{i_1,i_2,\ldots,i_s}$ being zero if  $\bigcap_{t=1}^s {\mathcal G}_{p_{i_t}}^b=\emptyset$, and being $(-1)^{s-1}$ otherwise. In other terms, putting $\Pi_k:=\prod_{i=2}^k p_i$,
$$\delta  \left( \bigcup_{j<k} {\mathcal G}_{p_j}^b \right)=-\sum_{\substack{m \mid \Pi_k\\ m  \in \mathcal{R}_b}} \frac{(-1)^{\omega(m)}}{\textrm{lcm} ( m ,\lambda(m))}$$
% $$ \delta(\mathcal{M}_{f_{1,b}})- \sum_{\substack{m \mid \Pi\\ m  \in \mathcal{R}_b}} \frac{(-1)^{\omega(m)}}{\textrm{lcm} ( m \lambda(m))}< \epsilon$$
where, $\omega(m)$ is the number of distinct prime factors of $m$.

%%\frac{(-1)^{\omega(m)}{\textrm{lcm} ( m \lambda(m))}
In the case $b=\pm1$ the asymptotic density of  $\mathcal{M}_{f_{_{1,b}}}$is closely related to that of the set $\mathfrak{P}:=\{n \in \mathbb{N}\textrm{ odd } : S_{\frac{n-1}{2}} \equiv 0\textrm{ (mod n)} \}$ which was defined and studied in \cite{nos2}. In this previous work $\delta(\mathfrak{P})$ was computed up to 3 digits: 0.379... More specifically, it was seen that $\delta(\mathfrak{P}) \in [0.379005, 0.379826]$.

\begin{prop}
For  $b\in \{-1,1\}$ the following holds:
$$\delta(\mathcal{M}_{f_{_{1,b}}})=2 \delta(\mathfrak{P})-\frac{1}{4} \in [0.50801, 0.50966].$$
\end{prop}
\begin{proof} Let $\mathbb{I}$ denote the set of odd positive integers. For any odd prime $p$ let us define the following set:
$$\mathcal{F}_p := \{p^2\ \textrm{(mod $2p(p-1)$)}\},$$
and recall the definition:
$$\mathcal{G}_p^b:=\mathbb{N} \cap \{-bp\ \textrm{(mod $p(p-1)$)}\}.$$
In \cite{nos2} it was seen that:
$$\mathbb{I}\setminus \mathfrak{P} = \bigcup_{p\geq3} \mathcal{F}_p$$
and in the previous proposition we have just proved that:
 $$\mathbb{N}\setminus \mathcal{M}_{f_{_{1,b}}}=\bigcup_{p\geq3} \mathcal{G}_p^b \cup \{(2\ \textrm{(mod 4)}\}.$$
If $b=\pm 1$ and for every prime $p$ we have $\delta(\mathcal{G}_p^b)= 2 \delta(\mathcal{F}_p)$. Reasoning in a way similar to that in \cite{nos2}, we can see that for every odd $b$ and every set of primes $\mathcal{P}$ it holds:
$$2\delta \left(\bigcap_{p \in \mathcal{P}}\mathcal{F}_p\right) = \delta \left(\bigcap_{p \in \mathcal{P}}\mathcal{G}_p^b\right).$$
Consequently:
$$\delta (\mathbb{N}\setminus \mathcal{M}_{f_{_{1,b}}})= \frac{1}{4}+\delta\left(\bigcup_{p\geq3} \mathcal{G}_p^b\right)=\frac{1}{4}+ 2 \delta (\mathbb{I}\setminus \mathfrak{P}) =\frac{1}{4}+2\left(\frac{1}{2}- \delta (\mathfrak{P})\right)$$
and finally, since $\delta(\mathfrak{P})$ belongs to $[0.379005, 0.379826]$ we obtain that:
$$\delta ( \mathcal{M}_{f_{_{1,b}}})= 1- \delta (\mathbb{N}\setminus \mathcal{M}_{f_{_b}})=2 \delta (\mathfrak{P})-1/4 \in [0.50801, 0.50966].$$
\end{proof}

It is easy to observe that of $b$ is odd and $|b|>1$ then $\delta ( \mathcal{M}_{f_{_{1,b}}})>\delta ( \mathcal{M}_{f_{_{1,1}}})$. Moreover, if $b$ and $b'$ are odd with $|b|\neq |b'|$ and $b$ divides $b'$ then $\delta ( \mathcal{M}_{f_{_{1,b'}}})>\delta ( \mathcal{M}_{f_{_{1,b}}})$. In addition it is also easy to observe that the supremum of the densities  $\delta ( \mathcal{M}_{f_{_{1,b}}})$ is:

$$\mathfrak{S}:=\lim_{k\rightarrow \infty} \sum_{m \mid \Pi_k} \frac{(-1)^{\omega(m)}}{\textrm{lcm} ( m \lambda(m))}-\frac14.$$

since this is a decreasing sequence any value of $k$ will provide an upper bound for $\mathfrak{S}$. Computing the value for $k=22$,  we can say that for every odd $b\neq\pm1$

$$ 0.50801<\delta ( \mathcal{M}_{f_{_{1,1}}})<\delta ( \mathcal{M}_{f_{_{1,b}}})<\mathfrak{S}<0.647.$$

We will say that a positive integer $n$ is an \emph{anti-Korselt number} if for every $p$ prime divisor of $n$, $p-1$ does not divide $(n-1)$. This section will be closed computing the asymptotic  density of anti-Korselt numbers. In order to do this we observe that Theorem \ref{TI} gives the following characterization.

\begin{lem}
An integer $n$ is an anti-Korselt number if and only if $\displaystyle{\sum_{j=1}^{n} j^{n-1} \equiv 0\textrm{ (mod n)}}$ and $4\nmid n$.
\end{lem}
\begin{proof}
Just apply Theorem \ref{TI} and observe that, by definition, anti-Korselt numbers are odd.
\end{proof}

\begin{prop}
The set of anti-Korselt numbers has asymptotic density whose value is:
$$ 2\delta(\mathfrak{P})-\frac{1}{2} \in [0.25801, 0.259652].$$
\end{prop}
\begin{proof}
By the previous lemma the set of anti-Korselt numbers is  $\mathfrak{K}:=\mathcal{M}_{f_{1,-1}} \setminus 4 \mathbb{N}$. Since, $4\mathbb{N} \subset\mathcal{M}_{f_{1,-1}}$, it follows that
$$\delta (\mathfrak{K})=\delta(\mathcal{M}_{f_{_{1,-1}}})-\frac{1}{4}=2 \delta(\mathfrak{P})-\frac{1}{2}$$
as claimed.
\end{proof}

\section{$\mathcal{M}_f$ containing the prime numbers}
In this section we will characterize the set $\mathcal{M}_f$ for some functions $f$ such that $f(p)=\frac{p-1}{2}$ for every odd prime. Note that in this case $\mathcal{M}_f$ contains all odd primes. In particular, we will focus on $f=\frac{\varphi}{2}$ and $f=\frac{\lambda}{2}$, where $\varphi$ and $\lambda$ denote Euler and Carmichael function, respectively.
%Asimismo estudiaremos los casos de $f=\varphi\pm1$ and  $f=\varphi\pm2$.

\begin{prop}
$\mathcal{M}_{\frac{\varphi}{2}}=\{p^k: p\ \textrm{odd prime}\}$.
\end{prop}
\begin{proof}
If $p$ is an odd prime and $k\in\mathbb{N}$, $\frac{\varphi(p^k)}{2}=\frac{p^{k-1}(p-1)}{2}$ and $\gcd\left(\frac{p^{k-1}(p-1)}{2},p-1\right)<p-1$. Consequently we can apply Proposition 1 to get that $p^k\in\mathcal{M}_{\frac{\varphi}{2}}$.

Now, if $n$ is odd and there exists $p,q$ distinct odd primes dividing $n$ it readily follows that $p-1$ divides $\frac{\varphi(n)}{2}$ so Proposition 3 applies and it follows that $n\not\in\mathcal{M}_{\frac{\varphi}{2}}$. Thus, if an odd $n\in\mathcal{M}_{\frac{\varphi}{2}}$ it must be $n=p^k$.

Finally, if $n\in\mathcal{M}_{\frac{\varphi}{2}}$ is even Proposition 4 implies that 4 divides $n$ and Proposition 6 implies that $\frac{\varphi(n)}{2}$ is odd. Since these statements are contradictory the result follows.
\end{proof}

In what follows we will use the notation $\mathcal{E}(m):=\max\{k \in \mathbb{N}: 2^k\ \textrm{divides }  m\}$.

\begin{prop}
Let $n=2^mp_{1}^{r_1} \cdots p_{s}^{r_s}$ with $s>0$. Then $n\in \mathcal{M}_{\frac{\lambda}{2}}$ if and only if one of these conditions holds:
\begin{itemize}
\item[i)] $m=0$ and $\mathcal{E}(p_i-1)=\mathcal{E}(p_j-1)$ for every $i,j$.
\item[ii)] $m=2$ or $3$, $\mathcal{E}(p_i-1)=1$ for every $i$ and $\frac{n}{2^m}\neq 3$.
\end{itemize}
\end{prop}
%\begin{prop}
%Let $n=2^mp_{1}^{r_1} \cdots p_{s}^{r_s}$ with $s>0$. For every $i$ consider $m_i$ such that $\frac{p_{i}-1}{2^{m_i}}$ is odd. Then
%$n\in \mathcal{M}_{\frac{\lambda}{2}}$ if and only if one of these conditions holds:
%\begin{itemize}
%\item[i)] $m=0$ and $m_i=m_j$ for every $i,j$.
%\item[ii)] $m=2$ or $3$, $m_i=1$ for every $i$ and $\frac{n}{2^m}\neq 3$.
%\end{itemize}
%\end{prop}
\begin{proof}
If condition i) holds, $n=p_1^{r_1}\cdots p_s^{r_s}$ and $p_i=2^tq_i+1$ with $q_i$ even and $t$ not depending on $i$. In this case $\lambda(n)=\textrm{lcm}(\varphi(p_1^{r_1}),\dots,\varphi(p_s^{r_s}))=2^t\textrm{lcm}(p_1^{r_1-1}q_1,\dots,p_s^{r_s-1}q_s)=2^tL$ with $L$ odd. Consequently $\frac{\lambda(n)}{2}=2^{t-1}L$ and since $L$ is odd it follows that $p_i-1$ does not divide $\frac{\lambda(n)}{2}$ and Proposition 3  implies that $n\in\mathcal{M}_{\frac{\lambda}{2}}$.

If condition ii) holds, it follows that $\lambda(n)=2L$ with $L>1$ odd. Consequently $\frac{\lambda(n)}{2}=L>1$ is odd and Proposition 5 applies to conclude that $n\in\mathcal{M}_{\frac{\lambda}{2}}$.

Finally, assume that $n=2^mp_{1}^{r_1} \cdots p_{s}^{r_s}$ with $s>0$ and $p_i=2^{m_i}q_i+1$ with $q_i$ odd is such that $n\in\mathcal{M}_{\frac{\lambda}{2}}$. First of all, Proposition 4 implies that $m=0$ or $m>1$.

If $m>1$, Proposition 5 (i) implies that $\frac{n}{2^m}\neq 3$ and Proposition 6 implies that $\frac{\lambda(n)}{2}$ is odd so $m=2$ or $3$ and $p_i^{r_i-1}(p_i-1)=\varphi(p_i^{r_i})=2L_i$ with $L_i$ odd; i.e., $p_i-1=2q_i$ with $q_i$ odd as claimed.

If, on the other hand, $m=0$, Proposition 3 implies that $p_i-1$ does not divide $\frac{\lambda(n)}{2}$ for any $i$. But if $m_i>m_j$ for some $i\neq j$ we have that $2^{m_i-1}q_j$ divides $\frac{\lambda(n)}{2}$ and, consequently, $p_j-1$ divides $\frac{\lambda(n)}{2}$. A contradiction.
\end{proof}

Before we proceed we will introduce some notation and technical results. Given a prime $p$ and a subset $A\subseteq\mathbb{N}$, we define the set:
$$A_{p}:=\{n\in A: p\mid n\ \textrm{but}\ p^{2}\nmid n\}.$$

With this notation we have the following result.

\begin{lem}
If for a set of primes $\{p_{i}\}_{i\in I}$ we have $\delta(A_{p_{i}})=0$ for every $i\in I$, and $\sum\limits_{i\in I}p_{i}^{-1}=\infty,$ then
$\delta(A)=0$.
\end{lem}

Now, given a positive integer $k$ we define the set:
$$\Upsilon_k:=  \{n\textrm{ odd }: \mathcal{E}(p-1)=k\ \textrm{for every $p\mid n$}\}.$$
With this notation, Proposition 11 states that:
$$\mathcal{M}_{\frac{\lambda}{2}}=\left(\bigcup_{k=1}^\infty \Upsilon_k \cup 4 \Upsilon_1 \cup 8 \Upsilon_1\right) \setminus\{12,24\}.$$

We are in the condition to compute the asymptotic density of $\mathcal{M}_{\frac{\lambda}{2}}$.

\begin{prop}
$\mathcal{M}_{\frac{\lambda}{2}}$ has zero asymptotic density.
\end{prop}
\begin{proof}
Since
$$\mathcal{M}_{\frac{\lambda}{2}}=\left(\bigcup_{k=1}^\infty \Upsilon_k \cup 4 \Upsilon_1 \cup 8 \Upsilon_1\right) \setminus\{12,24\},$$
it will be enough to show that $A=\bigcup_{n=1}^\infty \Upsilon_n$ has zero asymptotic density.

For any prime $p$ let us introduce the following sets:
$$
\mathcal{I}_{p}:=\{p\}\cup\{q\ \textrm{prime}: \mathcal{E}(p-1)\neq\mathcal{ E}(q-1)\},
$$
$$
\mathcal{T}_{p}:=\{n\in \mathbb{N}: p\nmid n\},
$$
$$\mathcal{K}_{p}:=\{pk: k\in \mathbb{N}\}.$$

It is easy to observe that, for any prime $p$:
$$\mathbb{N}\setminus A_{p}=\bigcup_{q\in \mathcal{I}_{p}}p\mathcal{K}_{q}\cup \mathcal{T}_{p}.$$

Now, considering that
$$\mathcal{I}_{p}=\{p\}\bigcup \{ q\ \textrm{prime}: q \not\equiv 2^{k}+1\ \textrm{(mod $2^{k+1}$) with}\ k=\mathcal{E}(p-1)\}$$
it is clear that $\displaystyle{\delta\left(\bigcup_{q\in \mathcal{I}_{p}} \mathcal{K}_{q}\right)=1}$. Thus, $\displaystyle{\delta\left(\bigcup_{q\in \mathcal{I}_{p}} p\mathcal{K}_{q}\right)=\dfrac{1}{p}}$ and since $\delta(\mathcal{T}_{p})=\dfrac{p-1}{p}$ it follows that, for any prime $p$:
$$\delta(A_p)=1- \delta(\mathbb{N}\setminus {A}_{p})=1-\delta\left(\bigcup_{q\in\mathcal{I}_{p}} p \mathcal{K}_{q}\right)-\delta(\mathcal{T}_{p})=0$$
and the result follows from the previous lemma.
\end{proof}


\begin{thebibliography}{1}

\bibitem{BOR}
D.~Borwein, J.~M. Borwein, P.~B. Borwein, and R.~Girgensohn.
\newblock Giuga's conjecture on primality.
\newblock {\em Amer. Math. Monthly}, 103(1):40--50, 1996.

\bibitem{CARLITZ}
L.~Carlitz.
\newblock The {S}taudt-{C}lausen theorem.
\newblock {\em Math. Mag.}, 34:131--146, 1960/1961.

\bibitem{GIUGA}
G.~Giuga.
\newblock Su una presumibile propriet\'a caratteristica dei numeri primi.
\newblock {\em Ist. Lombardo Sci. Lett. Rend. Cl. Sci. Mat. Nat. (3)},
  14(83):511--528, 1950.

\bibitem{nos1}
J.M. Grau and A.M. Oller-Marc\'{e}n.
\newblock Generalizing giuga's conjecture.
\newblock {\em preprint}, arXiv:1103.3483v1 [math.NT].

\bibitem{nos2}
J.M. Grau,  F. Luca and A.M. Oller-Marc\'{e}n.
\newblock On a variant of giuga numbers.
\newblock {\em Acta Math. Sin. (Engl. Ser.)}, 28(4):653--660, 2011.

\bibitem{HARDY}
G.~H. Hardy and E.~M. Wright.
\newblock {\em An introduction to the theory of numbers}.
\newblock Oxford University Press, Oxford, sixth edition, 2008.
\newblock Revised by D. R. Heath-Brown and J. H. Silverman, With a foreword by
  Andrew Wiles.

\bibitem{MAC}
K.~MacMillan and J.~Sondow.
\newblock Proofs of power sum and binomial coefficient congruences via
  {P}ascal's identity.
\newblock {\em Amer. Math. Monthly}, 118(6):549--551, 2011.

\bibitem{NIVEN}
I.~Niven.
\newblock Sets of integers of density zero.
\newblock In {\em Proceedings of the {I}nternational {C}ongress of
  {M}athematicians.}, page 298, Cambridge, 1950.

\end{thebibliography}
\end{document}